\documentclass[12pt]{amsart}
\usepackage{amscd,amssymb,graphics,color,a4wide,hyperref,verbatim}

\usepackage{comment}
\usepackage{multicol, fullpage} 
\usepackage[usenames,dvipsnames]{xcolor} 
\usepackage{tkz-euclide}
\usepackage{tikz, tikz-3dplot, pgfplots}
\usepackage{tikz-cd}
\usetikzlibrary[positioning,patterns] 
\usetikzlibrary{matrix,arrows,decorations.pathmorphing}

\usepackage{graphicx}
\usepackage{psfrag}
\usepackage{marvosym}
\usepackage{amsmath}
\usepackage{amsfonts}
\usepackage{amssymb}
\usepackage{amsthm}
\usepackage{mathrsfs}
\usepackage{enumitem} 
\usepackage{bbm}

\usepackage{ytableau}

\input xy
\xyoption{all}

\usepackage{calligra}
\usepackage{mathrsfs}
\DeclareMathAlphabet{\mathcalligra}{T1}{calligra}{m}{n}
\DeclareFontShape{T1}{calligra}{m}{n}{<->s*[1.5]callig15}{}

\newtheorem{theorem}{Theorem}[section]

\newtheorem{lemma}[theorem]{Lemma}

\newtheorem{corollary}[theorem]{Corollary}

\theoremstyle{definition}
\newtheorem{definition}[theorem]{Definition}

\newtheorem{remark}[theorem]{Remark}

\newtheorem{theorem-definition}[theorem]{Theorem-Definition}

\numberwithin{equation}{section}

\newcommand{\CC} {\mathbb{C}}

\newcommand{\RR} {\mathbb{R}}

\newcommand{\ZZ} {\mathbb{Z}}

\newcommand {\shD} {\mathcal{D}}

\newcommand {\shH} {\mathcal{H}}
\newcommand {\shI} {\mathcal{I}}

\newcommand {\shS} {\mathcal{S}}

\newcommand {\shP} {\mathcal{P}}

\newcommand {\fok}  {\mathfrak{k}}

\newcommand {\fot}  {\mathfrak{t}}

\newcommand{\sExt}{\mathscr{E} \kern -1pt xt}

\newcommand {\sHom}{\mathscr{H}\kern-5pt\mathcalligra{om}}

\renewcommand {\Im} {\operatorname{Im}}

\renewcommand {\ker } {\operatorname{Ker}}
\newcommand {\Ker} {\operatorname{Ker}}

\newcommand {\rank} {\operatorname{rank}}

\newcommand{\Diff}{\operatorname{Diff}}

\newcommand{\sTor}{\mathscr{T} \kern -3pt or}

\newcommand {\vol} {\operatorname{vol}}

\begin{document}

\title{Quantum spaces associated to mixed polarizations and their limiting behavior on toric varieties
\\
\author{Dan WANG} 
}

\address{CAMGSD and Department of Mathematics, IST\\ University of Lisbon}
\email{dan.wang@tecnico.ulisboa.pt; dwang116@link.cuhk.edu.hk}

\date{}
\maketitle

\begin{abstract} Let $(X, \omega, J)$ be a toric variety of dimension $2n$ determined by a Delzant polytope $P$. As indicated in \cite{LW1}, $X$ admits a natural mixed polarization $\shP_{k}$, induced by the action of a subtorus $T^{k}$.
In this paper, we first establish the quantum space $\shH_{k}$ for $\shP_{k}$, identifying a basis parameterized by the integer lattice points of $P$. This confirms that the dimension of $\shH_{k}$ aligns with those derived from K\"ahler and real polarizations. Secondly, we examine a one-parameter family of Kähler polarizations $\shP_{k,t}$, defined via symplectic potentials, and demonstrate their  convergence to $\shP_{k}$. 
Thirdly, we verify that these polarizations $\shP_{k,t}$ coincide with those induced by imaginary-time flow. Finally, we explore the relationship between the quantum space $\shH_{k,0}$ and $\shH_{k}$, establishing that ``$\lim_{t\rightarrow \infty}\shH_{k,t}=\shH_{k}$".
\end{abstract}

\section{Introduction}
An important problem in geometric quantization is to explore the dependence of the quantum spaces on the choice of polarizations on a symplectic manifold $(X,\omega)$. Assume symplectic manifold $X$ admits a pre-quantum line bundle $(L,\nabla, h)$, which is a complex line bundle $L \rightarrow X$ with Hermitian metric $h$ and a Hermitian connection $\nabla$ such that the curvature form $F_{\nabla}=-i\omega$. A polarization $\shP$ on $X$ is an integrable Lagrangian subbundle of $TX\otimes \CC$. The quantum Hilbert space $\shH_{\shP}$ associated to a polarization $\shP$ is the subspace of $\Gamma(X,L)$ defined by:
\begin{equation}\shH_{\shP} = \{ s \in \Gamma(X, L) \mid \nabla_{\xi} s =0, \forall \xi \in \Gamma(X, \shP)\} .
\end{equation}

Another essential problem in geometric quantization is determining the quantum space $\shH_{\shP}$ associated to a polarization $\shP$.
There are three types of  polarizations on $X$ has been studied, K\"ahler polarization, real polarization, and mixed polarization (see Appendix \ref{ap-pol}). 
 K\"ahler polarization $\shP_{J}$ comes from an integral complex structure $J$ on $X$ such that $(X,\omega,J)$ is a K\"ahler manifold. In this case, the quantum space $\shH_{\shP_{J}}$ is the space of $J$-holomorphic sections. 
 Real polarization $\shP_{\RR}$ satisfying $\shP_{\RR} = \bar{\shP}_{\RR}$. A completely integrable system $f: M \rightarrow \RR^{n}$ gives a singular real polarization $\shP_{\RR} = \ker df$. In this case, the quantum space $\shH_{\shP_{\RR}}$ consists of distributional sections which are covariant constant along fiber with supports on Bohr-Sommerfeld fibers. 
Leung and Wang showed that the existence of mixed polarization $\shH_{\mathrm{mix}}$ on K\"ahler manifold with $T$-symmetry in \cite{LW1} and give an geometric description of the weight $\lambda$- quantum subspace $\shH_{\mathrm{mix},\lambda}$, for regular $\lambda$ in \cite{LW2}.


Assume $(X,\omega, J)$ is a $2n$-dimensional  toric variety admits a Hamiltonian torus $T^{n}$ action together with moment map $\mu_{P}: X \twoheadrightarrow P$. Then $X$ admits both K\"ahler polarization $\shP_{J}=T^{0,1}X$ with respect to $J$ and real polarization $\shP_{\RR}=\ker d \mu_{P}$ on $X$. From a physical perspective, the quantum space should be independent of the polarization choice. Specifically, the dimension of the quantum spaces should remain invariant. A result, often attributed to Danilov and Atiyah \cite{Da, GGK}, asserts that the number of integer points in $P$, which correspond to the images under $\mu_{P}$ of the Bohr-Sommerfeld fibers of the real polarization $\shP_{\RR}=\mathrm{Ker} d\mu$, equals the dimensionality of $H^{0}(X,L)$.

In this paper, we explore the mixed polarization $\shP_{k}$, as identified in \cite{LW1}, on toric variety $(X, \omega, J)$ associated with subtorus $T^{k}$-action. We begin by determining the quantum space $\shH_{k}$ associated to the polarization $\shP_{k}$.
 Especially, We have identified a basis for the quantum space $\shH_{k}$ that is parameterized by the integer points of the moment polytope.
This implies that the dimension of quantum space $\shH_{k}$, associated to the polarization $\shP_{k}$, matches the dimension of the quantum spaces associated to both K\"ahler polarization and real polarization.
 Additionally, we study a one-parameter family of K\"ahler polarizations, defined by potentials, which converges to $\shP_{k}$. We then demonstrate the conditions under which the K\"ahler polarizations, derived from this one-parameter family of symplectic potentials, coincide with the family induced by imaginary-time flow, as investigated in \cite{LW1}. Finally, we elucidate the relationship between quantum spaces associated to K\"ahler polarization $\shP_{k,0}$ and the mixed polarization $\shP_{k}$ by show that ``$\lim_{t\rightarrow \infty}\shH_{k,t}=\shH_{k}$".
Throughout this paper, we assume the following:
\begin{enumerate}
\item[$(*)$]: Let $(X, \omega, J)$ be a 2n-dimensional toric variety determined by a Delzant polytope $P$  with an associated moment map $\mu_{P}: X \twoheadrightarrow P \subset (\fot^{n})^{*}$.
 We consider the Hamiltonian $k$-dimensional subtorus action $\rho_{k}: T^{k} \rightarrow \Diff(X, \omega, J)$, and denote the corresponding moment map by $\mu_{k}: X \twoheadrightarrow \Delta_{k}  \subset (\fot^{k})^{*}$. Additionally , we pick a strictly convex function $\varphi_{k}: (\fot^{k})^{*} \rightarrow \RR$.
\end{enumerate}

In the process of geometric quantization, selecting a polarization is a fundamental step.  According to \cite{LW1}, given assumption $(*)$, a polarization $\shP_{k}$ can be constructed as follows:
 \begin{equation}
 \shP_{k}=(P_{J} \cap \shD^{k}_{\CC}) \oplus \shI^{k}_{\CC},
 \end{equation}
 where $\shD^{k}_{\CC}=(\Ker d\mu_{k}) \otimes \CC$ and ~$\shI^{k}_{\CC}= (\Im d\rho_{k}) \otimes \CC$.
When $k=n$, on the open dense subset $\mathring{X}$ of $X$, $\shP_{n}$ coincides with the real polarization $\shP_{\RR}=\Ker d\mu_{P}$ studied in \cite{BFMN}. For $1\le k < n$, $\shP_{k}$ represents a (singular) mixed polarization with $\mathrm{rk}_{\RR}({\shP_{k}}|_{\mathring{X}})=k$ (refer to \cite[Theorem 1.1]{LW1}).
The pre-quantum line bundle $(L, \nabla, h)$ is determined by Delzant polytope $P$ (see subsection \ref{pre-data}). To define the quantum space $\shH_{k}$ associated to $\shP_{k}$, it is necessary to embed the space of smooth sections $\Gamma(X, L)$ into the space of distributional sections $\Gamma_{c}(X,L^{-1})'$ using the Liouville measure $\vol_{X}=\frac{\omega^{n}}{n!}$. The quantum space $\shH_{k}$ is then defined by:
\begin{equation}
\shH_{k}= \{ \delta \in \Gamma(X, L^{-1})' \mid \nabla_{\xi}\delta  =0, \forall \xi \in \Gamma(X, \shP_{k})\}.
\end{equation}

  Letting $P_{\ZZ}$ denote $P \cap (\fot^{n})^{*}_{\ZZ}$. Inspired by \cite[Theorem 3.12]{LW2} and \cite[Theorem 1.1]{BFMN}, for each $m \in P_{\ZZ}$, we defined $\delta_{k}^{m}$ due to stratification theorem by Sjamaar and Lerman \cite{SL}.
 \begin{definition}[Definition \ref{def3-1}]For any $m\in P_{\ZZ}$, we define {\em the distributional section} $\delta_{k}^{m} \in \Gamma_{c}(X,L^{-1})'$ associated to $\sigma_{k,0}^{m}$ and $\mu_{k}$ by: for any test section $\phi \in \Gamma_{c}(X,L^{-1})$ \begin{equation}
 \delta_{k}^{m} (\phi) =\frac{1}{ c^{m}_{k}} \int_{M^{q}_{k}}  \langle \sigma_{k,0}^{m}, \phi \rangle |_{M^{q}_{k}} \vol_{k}^{q},
 \end{equation}
where $c_{k}^{m}$ is a constant defined by: $c_{k}^{m} = \int_{M^{q}_{k}} \|\sigma_{k,0}^{m}|_{M^{q}_{k}}\| \vol_{k}^{q}.$
\end{definition}
Our first result establishes that there exists a basis for the quantum space $\shH_{k}$ parameterized by the integer point in $\shP$. Consequently, the dimension of $\mathrm{\shH_{k}}$ coincides to those of the quantum spaces associated to both K\"ahler and real polarizations.

\begin{theorem}[Theorem \ref{thm5-0}]
Under assumption $(*)$, $\{\delta_{k}^{m}\}_{m\in P_{\ZZ}}$ forms a basis of quantum space $\shH_{k}$. In particular, $\mathrm{dim}\shH_{k}=\mathrm{dim}H^{0}(X,L)$.
\end{theorem}  
When $k=n$, $\shH_{n}$ coincides with $\shH_{\RR}$ studied in \cite{BFMN}. Moreover, the above theorem coincides with \cite[Theorem1.1]{BFMN}.
 Guillemin \cite{Gui2} established that the canonical toric K\"ahler metric and symplectic form on $(X, \omega, J)$ are determined by a symplectic potential $g_{0}: P \rightarrow \RR$. Abreu \cite{Ab1,Ab2} expanded on this by describing all toric complex structures through symplectic potentials of the form $g=g_{0}+ \varphi$, which adhere to a specific convexity condition (see Definition \ref{def-po}). Considering any $k$-dimensional subtorus $T^{k}$ of $T^{n}$, let $i_{k}^{*}$ represent the dual map corresponding to the inclusion $i_{k}: t^{k} \rightarrow t^{n}$ of the Lie algebras.
Drawing on the insights from \cite{BFMN} and \cite{HK}, we identified a one-parameter family of toric complex structures $J_{k,t}$, given by the family of symplectic potentials:
\begin{equation}
g_{k,t} = g_{0}+ t(\varphi_{k}\circ i_{k}^{*}): P \rightarrow \RR,
\end{equation}
such that the corresponding K\"ahler polarizations $\shP_{k,t}$ degenerate to $\shP_{k}$, as $t \rightarrow \infty$. 
\begin{theorem} [Theorem \ref{thm4-3}]
Under the assumption $(*)$, 
we have 
$$\lim_{t \rightarrow \infty} \shP_{k,t} =  \shP_{k},$$
 where the limit is taken in the positive Lagrangian Grassmannian of the complexified tangent space at each point in $X$.  
\end{theorem}

When $k=n$, the family of K\"ahler polarizations $\shP_{n,t}$ coincides with the family of K\"ahler polarizations $\shP_{\CC}^{s}$ studied in \cite[Theorem 1.2]{BFMN}.
Let $\hat{\varphi}_{k}: \Delta_{k} \rightarrow \RR$ be another strictly convex function. Let $X_{\hat{\varphi}_{k}}$ be the Hamiltonian vector field on $X$ associated to the function $\hat{\varphi}_{k}\circ i_{k}^{*} \circ \mu_{P}$. 
 According to \cite[Theorem 3.19, 3.20, and 3.21]{LW1}, we are able to construct a one-parameter family of K\"ahler polarizations $\shP_{k,t}'$ using imaginary-time flow $e^{-itX_{\hat{\varphi}_{k}}}$.

 Our next result demonstrate that when $\varphi_{k}=\hat{\varphi}_{k}$ and $\shP_{k,0}=\shP_{k,0}'$, these two one-parameter families of K\"ahler polarizations coincide.

\begin{theorem}[Theorem \ref{thm3-0-4}]
Let $\{\mathcal{P}_{k,t}\}$ be a one-parameter family of K\"{a}hler polarizations given by the symplectic potentials $g_{k,t} = g_0 + \varphi_k \circ i_k^*$, and let $\{\mathcal{P}'_{k,t}\}$ be another one-parameter family of K\"{a}hler polarizations derived using the imaginary-time flow $e^{-itX_{\hat{\varphi}_k}}$. If $\shP_{k,0}=\shP_{k,0}'$ and $\varphi_{k}=\hat{\varphi}_{k}$, then
$$\mathcal{P}_{k,t} = \mathcal{P}'_{k,t}, ~\forall t\ge 0.$$
\end{theorem}

To establish this equivalence, we used the formula for the K\"ahler potential given by Leung and  Wang in \cite{LW3}, the relationship between K\"ahler potential and moment map studied by Burns and Guillemin in \cite{BG}, and Legendre transformation described by Abreu in \cite{Ab2}.

\begin{corollary}[Corollary \ref{com-family}]
Let $\{J_{k,t}\}$ be the one-parameter family of toric complex structures given by the potential $g_{k,t} = g_0 + \varphi_k \circ i_k^*$, and let $\{J'_{k,t}\}$ be another one-parameter family of complex structures derived using imaginary-time flow $e^{-itX_{\hat{\varphi}_k}}$. If $J_{k,0}=J_{k,0}'$ and $\varphi_{k}=\hat{\varphi}_{k}$, then $$J_{k,t} = J'_{k,t}, ~\forall t\ge 0.$$
\end{corollary}

 Our final result aims to elucidate the relationship between the quantum spaces $\shH_{k,0}$ and $\shH_{k}$ along the one-parameter family of K\"ahler polarizations $\{\shP_{k,t}\}_{t\ge0 }$. Specifically, we aim to demonstrate that 
``$\lim_{t\rightarrow \infty}\shH_{k,t}=\shH_{k}$" 
in the following sense. 
\begin{theorem}[Theorem \ref{thm5-1}]
Under assumption $(*)$, for any strictly convex function $\varphi_{k}$ in a neighborhood of $\Delta_{k}$ and $m \in P \cap \fot^{*}_{\ZZ}$, consider the family of $L^{1}$-normalized $J_{k,t}$-holomorphic sections $\frac{\sigma_{k,t}^{m}}{\|\sigma_{k,t}^{m}\|_{L_{1}}}$, under the injection $i: \Gamma(X, L) \rightarrow \Gamma_{c}(X,L^{-1})'$.
We observe that:
\begin{equation}
i\left(\frac{\sigma_{k,t}^{m}}{\|\sigma_{t}^{m}\|_{L_{1}}}\right) \longrightarrow \delta_{k}^{m}, \text{as}~ t \rightarrow \infty.
\end{equation}
In the sense that for any $\phi \in \Gamma_{c}(X, L^{-1})'$, we have:
\begin{equation}
\lim_{t \rightarrow \infty} \int_{X} \langle \frac{\sigma_{k,t}^{m}}{\|\sigma_{k,t}^{m}\|_{L_{1}}}, \phi \rangle e^{\omega} = \frac{1}{c_{k}^{m}}\int_{X_{k}^{q}} \langle \sigma_{k,0}^{m}, \phi \rangle |_{\mu_{k}^{-1}(0)} \vol_{k}^{q},
\end{equation}
where $q=i_{k}^{*}(m)$ and
$c_{k}^{m} = \int_{X^{q}_{k}} \|\sigma_{k,0}^{m}|_{X^{q}_{k}}\| \vol_{k}^{q}.$
\end{theorem}
\begin{remark} When $k=n$, the above theorem coincide with the result of \cite[Theorem 1.3]{BFMN}.
\end{remark}

The study of quantum spaces associated to (singular) real polarization has been demonstrated by Śniatycki \cite{S}, by Weitsman \cite{Wei1, Wei2}, by Jeffrey and Weitsman \cite{JW}, and by Hamilton and Miranda \cite{HM}.
Andersen \cite{An2} has investigated the quantum spaces associated with mixed polarizations.

The relationships between quantizations associated with different polarizations have been extensively explored. Using the space of holomorphic symmetric tensor on the moduli space of stable bundles over a Riemann surface, Hitchin in \cite{Hi} constructed a projectively flat connection on a vector bundle over Teichm\"uller space. Quantum spaces associated to K\"ahler polarizations converging to quantum spaces associated to real polarizations were carried out 
by Kirwin and Wu for symplectic vector spaces \cite{KW};
by Hall \cite{Hal} and Florentino, Matias, Mour\~ao and Nunes \cite{FMMN1,FMMN2} for cotangent bundles of Lie groups;
by Baier, Florentino, Mour\~ao and Nunes for toric varieties \cite{BFMN}; 
by Guillemin and Sternberg, Hamilton and Konno for flag varieties \cite{GS3,HK}. Quantum spaces associated to K\"ahler polarizations converging to quantum spaces associated to mixed polarizations were carried out by Baier, Ferreira, Hilgert, Kaya, Mourão and  Nunes \cite{BFHMN, BHKMN}.

There are many previous works by others on closely related problems for toric varieties \cite{ CS, KMN1,KMN4}, flag varieties \cite{BHKMN}, cotangent bundles of compact Lie groups \cite{KMN2, MNP}, toric degenerations \cite{HHK, HK3}, and so on \cite{An3,An4,An5, BFHMN,BBLU, GMW, GS1, HK1, HK2, Hi,LY1, MNR}.

\subsection{Acknowledgements} The author would like to express gratitude to Naichung Conan Leung for numerous helpful discussions, to Siye Wu for insightful discussions. Appreciation is also extend to Reyer Sjamaar for helpful discussion on stratification in quotient spaces and to Daniel Burns for helpful discussion on the relationship between K\"ahler potential and moment map. The author is grateful to Alejandro Cabrera and Ziming Nikolas Ma for helpful suggestions. Special thanks to Qingyuan Jiang for many useful discussions. This work was supported by CAMGSD UIDB/04459/2020 and UIDP/04459/2020.

 \section{Preliminaries}

 \subsection{Burns-Guillemin's theorem}
In this subsection, we review the equivariant Darboux theorem for K\"ahler forms on $N = U \times T_{\mathbb{C}}^{k}$, where $U$ is an open and convex subset of $\mathbb{C}^{n-k}$.
We also discuss the relationship between K\"ahler potentials and moment maps as established by Burns and Guillemin in \cite{BG}.

 The action of $T^{k}$ on $N$ is through standard multiplication of $T^{k}$ on $T^{k}_{\CC}$. Consider $w_{1}, \cdots ,w_{n-k}$ and $z_{1}, \cdots, z_{k}$ as the standard coordinate functions on $U$ and $T^{k}_{\CC}$, respectively. Let $\omega$ be a $T^{k}$-invariant K\"ahler form that is Hamiltonian with respect to the $T^{k}$ action.  
Burns and Guillemin in \cite{BG} demonstrated the following:
 $\omega= \sqrt{-1} \partial \bar{\partial} \rho$,
 where $\rho$ is a $T^{k}$-invariant function. Namely,
\begin{equation}\label{eq-po} \rho = \rho(w_{1},\cdots,w_{n-k},t_{1},\cdots,t_{k}), t_{i}=|z_{i}|^{2}. \end{equation}

\begin{theorem}\label{thm2-0} \cite[Theorem 3.1] {BG} Let $\rho$ a strictly plurisubharmonic functions of the form
(\ref{eq-po}). 
Let $\mu=(\mu_{1}, \cdots, \mu_{k})$ be the moment map associated with the action of $T^{k}$ on $N$. Under the change of variables $t_{j}=e^{s_{j}}$, where $t_{j}=|z_{j}|^{2}$, one has:
$$\mu_{j} = \frac{\partial}{\partial s_{j}} f(w_{1},\cdots,w_{n-k}, s_{1}, \cdots,s_{k} )=  t_{j} \frac{\partial}{\partial t_{j}} \rho(w_{1},\cdots,w_{n-k}, t_{1},\cdots,t_{k}).$$
where $f(w_{1},\cdots,w_{n-k}, s_{1}, \cdots,s_{k} )= \rho(w_{1},\cdots,w_{n-k}, t_{1},\cdots,t_{k})$.
\end{theorem}

\subsection{Prequantum data of toric varieties}\label{pre-data}
In this section, we revisit the toric K\"ahler structures of a toric variety $(X, \omega, J)$ determined by the Delzant polytope $P$ along with the Hamiltonian $T^{n}$-action and the moment map $\mu_{P}: X \twoheadrightarrow P$. The Lie algebra of $T^{n}$ is denoted by $\fot^{n}$. The construction of the Toric variety $X$ is recalled from both symplectic and complex perspectives (refer to \cite{Gui1} for the former and \cite{F}, \cite{CJH} for the latter), following the conventions established in \cite{HK}. These viewpoints are identified based on a selection of symplectic potentials as detailed in \cite{Ab1, Ab2, Gui1, Gui2} within subsection 2.3. Moreover, the polytope $P$ leads to the determination of an equivariant $J$-holomorphic line bundle $(L, h, \nabla)$ with curvature $-i \omega$ (refer to  \cite{Kos}).
Consider the bounded Delzant polytope $P$ defined as:
\begin{equation}\label{eq-poly}
P = \{ p \in (\fot^{n})^* ~|~ \langle p, r_j \rangle + \lambda_j \ge 0 , j=1, \dots, d \} \subset (\fot^{n})^{*},
\end{equation}
 where $\langle ~, ~ \rangle: (\fot^n)^* \times \fot^n \rightarrow \RR$ is the natural pairing and $r_j$ stands for a primitive vector in the lattice $\fot_{\ZZ}^{n} \subset \fot^{n}$ for $j=1, \dots, d$.
It  is assumed that $\lambda_1, \dots, \lambda_d \in \ZZ$. 
Define:
\begin{equation}\label{wall}
l_j(p)= \langle p, r_j \rangle + \lambda_j, ~~~~~~ F_j = \{ p \in (\fot^{n})^* ~|~ l_j(p) = 0 \} , j=1, \dots, d.
\end{equation}
Let $T^d$ be a real torus with the Lie algebra $\fot^d$ and let $X_1, \dots, X_d \in\fot^d_{\ZZ}$ be the basis of $\fot^d$. Let $u_1, \dots, u_d \in (\fot^{d})^*$ be the dual basis of $X_1, \dots, X_d \in\fot^d_{\ZZ}$.  
Define $\lambda_P= \lambda_1 u_1 + \dots + \lambda_d u_d \in (\fot^{d})^*_{\ZZ}$. Consider the surjective Lie algebra homomorphism $\pi: \fot^d \rightarrow \fot^n$ defined by $\pi(X_j)=r_j$ for $j=1, \dots, d$, which leads to the exact sequence:
\begin{equation}
0\rightarrow K \rightarrow T^{d} \rightarrow T^{n} \rightarrow 0.
\end{equation}
Here $K \subset T^{d}$ is a connected subtorus  with the Lie algebra $\fok$. Toric varieties will be constructed from the Delzant polytope in the forthcoming subsections \ref{symp-view} and \ref{comp-view}.

\subsubsection{Construct the toric variety $(X,\omega_{P})$ from a symplectic viewpoint}  \label{symp-view}
Let $\omega_{std}=\frac{i}{2} \sum_{j} dz_{j} \wedge d\bar{z}_{j}= \sum_{j} dx_{j} \wedge dy_{j}$ be the standard symplectic form on $\CC^d$, where $x_j$ and $ y_j$ are the real and imaginary parts of $z_j$ respectively.
The natural action of $T^d$ on $\CC^d$ induces a moment map 
\begin{equation}
\mu_{T^d}: \CC^d \rightarrow (\fot^d)^*, z \mapsto \mu_{T^d}(z_1, \dots, z_d)=\frac{1}{2}\sum_{j=1}^d|z_j|^2u_j.
\end{equation}

Let $\iota^*: (\fot^d)^* \rightarrow  \fok^*$ be the dual map of the inclusion $\iota :  \fok \rightarrow \fot^d$. Then moment map $\mu_K: \CC^d \rightarrow \fok^*$ for the action of the subtorus $K$ on $({\CC}^d, \omega_{std})$ is given by 
\begin{equation}\mu_{K}(z)=\frac{1}{2}\sum_{j=1}^d|z_j|^2 \iota^* u_j.
\end{equation}
The compact symplectic toric manifold $X$ is defined as the symplectic quotient $X=\mu_K^{-1}(\iota^* \lambda_P)/K$, and the symplectic structure $\omega_{P}$ is determined by $\pi^{*} \omega_{P}= \omega_{std}|_{\mu_K^{-1}(\iota^* \lambda_P)}$. 
The quotient torus $T^n=T^d/K$ acts on $(X, \omega_{P})$ with the moment map $\mu_{P}: X \twoheadrightarrow P \subset (\fot^n)^*$. This leads to the following diagram:
\[
\begin{tikzcd}[row sep=1.5em, column sep=3.5em]
1 \arrow{d}   && 0 \\
K \arrow{d}  \arrow[bend left]{r}& (\CC^{d},\omega_{std}) \arrow[equal]{d}{} \arrow{r}{\mu_{K}} & (\RR^{k})^{*} \arrow{u} \\
T^{d} \arrow{d} {\tilde{\pi}} \arrow[bend left]{r} &(\CC^{d},\omega_{std})  \arrow{r}{\mu_{T^{d}}}  &(\RR^{d})^{*} \arrow{u} {\iota^{*}} \\
T^{n} \arrow{d} \arrow[bend left]{r} & (X, \omega_{P}) \arrow{r}{\mu_{P}} & (\RR^{n})^{*} \arrow{u} {\pi^{*} }\\
1&&0\arrow{u}
\end{tikzcd}
\]
Take any $z \in \mu_{K}^{-1}(\iota^{*} \lambda_{P})$, we have $\iota^{*}(\mu_{T^d}(z)-\lambda_P)=0$. It turns out that 
$$ \mu_{T^d}(z)-\lambda_P \in \ker(\iota^* ) = \mathrm{Image} (\pi^* ),$$
and $\mu_{P}([z])$ is given by $\mu_{P}([z]) = (\pi^*)^{-1}(\mu_{T^d}(z)-\lambda_P) \in (\RR^{n})^{*}=(\fot^n)^*$. Note that, 
$$\mu_{P}(X_{P}) = (\pi^{*})^{-1}(\Im \mu_{T^{d}} \cap  (\iota^{*})^{-1}(\iota^{*} \lambda_{P}))=P.$$
Let $\tilde{L}_{symp}=\CC^d \times \CC$ be the trivial line bundle with the standard Hermitian metric $\tilde{h}$ and a Hermitian connection $\tilde{\nabla}$, where 
 \begin{equation}
\tilde{\nabla}=d - \sqrt{-1}\pi \sum_{i=j}^d (x_j dy_j -y_j dx_j), \text {and} ~F_{\tilde{\nabla}} = -i 2 \pi \omega.
\end{equation}
Let $\mathrm{Exp}_{T^d}:\fot^d \rightarrow T^d$ be the exponential map. Define the action of $T^d$ on $\tilde{L}_{symp}$ by: for any $(z,v) \in \CC^{d} \times \CC$ and $\xi \in \fot^{d}$,
$$(z,v)\mathrm{Exp}_{T^d} \xi = (z \mathrm{Exp}_{T^d}\xi, ve^{ \sqrt{-1} \langle \lambda_{P}, \xi \rangle})$$
where  $\lambda_P= \lambda_1 u_1 + \dots + \lambda_d u_d \in (\fot^{d})^*_{\ZZ}$. It's easy to check that this action preserves the Hermitian metric $\tilde{h}$ and Hermitian connection $\tilde{\nabla}$.
Then we define a prequantum line bundle $(\tilde{L}_{symp}, \tilde{h},\tilde{\nabla})$ on $(X, \omega_{P})$ as the quotient of the restriction of $\tilde{L}_{symp}$ to $\mu_K^{-1}(\iota^* \lambda_\Delta)$ by the action of the subtorus $K$.

It's straightforward to see that the quotient torus $T^n=T^d/K$ acts on $L_{symp}$, preserving $h$ and $\nabla$. 
Let $[z]_K \in X_{P}$ denote a point represented by $z \in \mu_K^{-1}(\iota^* \lambda_P)$, and let $[z,v]_K$ denote a point in $L_{symp}$ represented by $(z,v) \in \mu_K^{-1}(\iota^* \lambda_P) \times \CC$. Let $\mathring{X} = \mu_{P}^{-1}(\mathring{P})$ be the open dense subset of $X$, where $\mathring{P}$ is the interior of the Delzant polytope $P$. Then it is evident that 
$(\sqrt{2 l_1(p)}, \dots, \sqrt{2 l_d(p)}) \in \mu_K^{-1}(\iota^* \lambda_P)$ for any $p \in \mathring{P}$.
The map $\psi : \mathring{P} \times \fot^n/ \fot^n_\ZZ \rightarrow \mathring{X}$, defined by \begin{align}\label{symp-coord}
\psi (p, [q]) &=[(\sqrt{2l_1(p)}, \dots, \sqrt{2l_d(p)})]_K \mathrm{Exp}_{T^n}(q) \\
&= [(\sqrt{2l_1(p)} e^{ \sqrt{-1} \langle u_1, \tilde{q} \rangle}, \dots, \sqrt{2l_d(p)}e^{\sqrt{-1} \langle u_d, \tilde{q} \rangle})]_K \nonumber
\end{align}
is a diffeomorphism, 
where $\tilde{q} \in\fot^d$ is a lifting of $q$ such that $\pi(\tilde{q})=q$. 
It's easy to see that 

\begin{align*}
\mu_{P} \circ \psi(p, [q])&= \mu_{P}([(\sqrt{2l_1(p)} e^{ \sqrt{-1} \langle u_1, \tilde{q} \rangle}, \dots, \sqrt{2l_d(p)}e^{\sqrt{-1} \langle u_d, \tilde{q} \rangle})]_K)\\
&= (\pi^{*})^{-1}(l_{1}(p)-\lambda_{1}, \cdots, l_{d}(p)-\lambda_{d})=p.
\end{align*}
for any $(p, [q]) \in \mathring{P} \times \fot^n/\fot^n_{\ZZ}$. We begin by taking a trivialization of the prequantum line bundle $L_{symp}$ over
$\mathring{X}$ using a section $\mathring{\mathbbm{1}}$ of $L_{symp}|_{\mathring{X}}$. This section is defined as follows:
$$
\mathring{\mathbbm{1}}(p,[q])=[(\sqrt{2l_1(p)}, \dots, \sqrt{2l_d(p)}),1]_K \mathrm{Exp}_{T^n}(q) \in L_{symp}.
$$ 

Next, choose a $\ZZ$-basis $p_1, \dots, p_n \in (\fot^{n})^*_\ZZ$ and its dual basis $q_1, \dots, q_n \in \fot^n_\ZZ$. 
Let $\mathring{P}= \{ x=(x_1, \dots, x_n) \in \RR^n ~|~ \sum_{i=1}^n x_i p_i \in \mathring{P} \}$ be the interior of polytope $P$. 
This set up provides a symplectic coordinate $(x,[\theta])\in \mathring{P} \times \RR^n/\ZZ^n$ on $\mathring{P} \times\fot^n/\fot^n_\ZZ$, which can also be regarded as a coordinate on $\mathring{X}$ under the map $\psi$.

\subsubsection{Construct toric variety ($W_{P}, J)$ from a complex perspective}  \label{comp-view}
Let $P$ be a Delzant polytope defined as in \ref{eq-poly}, and let  $P^{0}$ be its set of vertices. 
Let $F_j \subset (\fot^{n})^*$ be the hyperplane defined by $l_{j} =0$, for $j=1, \dots, d$. We define $\CC^d_P$ as the union $\bigcup_{v \in P^{0}}\CC^d_v$, 
where 
\begin{equation}\CC^d_v = \{ z \in \CC^d ~|~ z_j \ne 0 ~~\text{if}~ j \in \{1, \dots, d \} \setminus \Lambda_v \},~ \text{with}~ \Lambda_v= \{ j ~|~ v \in F_j \}.
\end{equation}
We then define the compact complex toric variety $W_{P}$ as the quotient space $W_{P}=\CC^d_P/K_\CC$ with the induced complex structure $J$, where $K_\CC$ is the complexification of the subtorus $K$. 
Let $T^d_\CC$ be the complexification of the torus $T^d$. The quotient torus $T^n_{\CC}=T^d_{\CC}/K_{\CC}$ acts on $(W_{P}, J)$, and it is straightforward to see that the $T^{n}$-action preserves the complex structure $J$. Let $\tilde{L}_{comp}= {\CC}^d \times \CC$ be a trivial holomorphic line bundle on ${\CC}^d$. 
Define the action of $T^d_{\CC}$ on $\tilde{L}_{comp}$ as follows:
for any $(z,v) \in \CC^{d} \times \CC$ and $\xi \in \fot^{d}_{\CC}=\fot^{d} \otimes \CC$, 
\begin{equation}\label{eq-ac}(z,v)\mathrm{Exp}_{T^d_{{\CC}}}\xi = (z \mathrm{Exp}_{T^d_{\CC}} \xi, v e^{ \sqrt{-1} \langle \lambda_P, \xi \rangle} ),\end{equation}
 where $\lambda_P= \lambda_1 u_1 + \dots + \lambda_d u_d $.
 Then, we define a holomorphic line bundle $L_{comp}$ on $(W_{P}, J)$ as the quotient of the restriction of $\tilde{L}_{comp}$ to $\CC^d_P $ by the action of $K_\CC$. 
The quotient torus $T^n_\CC=T^d_\CC/K_\CC$ acts on $L_{comp}$, preserving its holomorphic structure $\bar{\partial}_{J}$. 
Let $[z]_{K_\CC}$ in $W_{P}$ denote a point represented by $z \in \CC^d_P$ and $[z,v]_{K_\CC}$ denote a point in $L_{comp}$ represented by $(z,v)$ in $\CC^d_P \times \CC$. There exists a meromorphic section $s_{\lambda}$ of $L_{comp}$ on $W_{P}$ associated with $\lambda_{P}$ defined by 
\begin{equation} 
s_{\lambda}([z]_{K_\CC}) =[z, \prod_{j=1}^d z_j^{\lambda_j}]_{K_\CC} \in L_{comp}~~ \text{for $ z \in \CC^d_P $}.
\end{equation} 
Note that $(\CC^\times)^{d}=\{ z \in \CC^{d} ~|~ z_i \ne 0 ~~ \text{for $i=1, \dots, d$} \} \subset \CC^d_P$. For $z\in (\CC^\times)^{d}$, we have $z_{j} \ne 0, j=1 \cdots d$. It implies that the section $s_{\lambda}$ is holomorphic and non-zero on $\mathring{W} = (\CC^\times)^d / K_\CC$. Therefore, $\mathring{s}_{\lambda} = s_{\lambda}|_{\mathring{W}}$ induces a holomorphic trivialization of $L_{comp}$ on $\mathring{W}$. For each $m \in P \cap (\fot^n)^*_{\ZZ}$, we define a section $\sigma^m$ of $L_{comp}$ by 
\begin{align}\label{def-ho-m}
\sigma^m ([z]_{K_\CC})=[z, \prod_{j=1}^d z_j^{l_j(m)}]_{K_\CC} \in L_{comp}~~ \text{for $ z \in \CC^d_P $.}
\end{align}
Since $l_{j}(m) \in \ZZ~ \text{for} ~j=1,\cdots,d$, we have that $\sigma^{m}$ is a homomorphic section. In fact, the set $\{ \sigma^m\}_{ m \in P \cap (\fot^{n})^*_{\ZZ} }$ forms a basis of the space of holomorphic sections $H^0(L_{comp}, \bar{\partial}_{J})$.
Now we provide a complex coordinate on the open dense subset $\mathring{W}$ of $W_{P}$. Choose a $\ZZ$-basis $p_1, \dots, p_n$ of $ (\fot^{n})^*_\ZZ$ and its dual basis $q_1, \dots, q_n$ in $\fot^n_\ZZ$, as described in Subsection \ref{symp-view}. 
We define a complex coordinate $\tilde{\psi}: \mathring{W} \rightarrow (\CC^\times)^n$ by 
\begin{equation}\label{comp-coord}
\tilde{\psi}([z]_{K_\CC})= (\prod_{j=1}^d z_j^{\langle p_1, r_j \rangle}, \dots, \prod_{j=1}^d z_j^{\langle p_n, r_j \rangle}),
\end{equation}
where $r_j \in\fot^n_\ZZ$ is the vector defined in \ref{eq-poly} for $j=1, \dots, d$. 
Since $\prod_{j=1}^d z_j^{\langle p_i, r_j \rangle}$ is a $K_\CC$-invariant meromorphic function on $\CC^d$, it descends to a meromorphic function on $W_{P}$.
If we set $(w_1, \dots,w_n)=\tilde{\psi}([z]_{K_\CC})$, then we have
\begin{align}\label{def-sigma-m}
\sigma^m & ([z]_{K_\CC}) 
=(\prod_{i=1}^n w_i^{\langle m, q_i \rangle })\mathring{s}_{\lambda}([z]_{K_\CC})~~\text{on $\mathring{W}$.} 
\end{align}

\subsection{Symplectic potentials on toric varieties}In Subsections \ref{symp-view} and \ref{comp-view}, we reviewed the construction of the symplectic toric manifolds $(X, \omega_{P})$ and the complex toric manifolds $(W_{P}, J)$ from a Delzant polytope $P$ defined in \ref{eq-poly} respectively.  
In this section, we identify them by canonical symplectic potentials $g_{0}$ due to Abreu \cite{ Ab1,Ab2} and Guillemin \cite{Gui1, Gui2} and denote $(X, \omega, J)$ as the toric variety determined by $P$ with symplectic structure $\omega$ and canonical complex structure $J$. The inclusion $\mu_K^{-1}(\iota^* \lambda_P) \subset \CC^d_P$ induces a diffeomorphism $\chi_{g_{0}}: X_{P} \rightarrow W_{P}$, given by the canonical symplectic potential $ g_{0}: \mathring{P} \rightarrow \RR$ as described in \cite{Gui1,Gui2}, with the following form:
\begin{equation} g_{0}(p)=\frac{1}{2 } \sum_{j=1}^d l_j(p) \log l_j(p)+ ( \text{ a linear function on $(\fot^{n})^*$ } ) ~~ \text{ for $ p \in \mathring{P}$ }.
\end{equation} 
It's easy to see that $ g_{0} $ can be extended continuously to the boundary of $P$. We fix a $\ZZ$-basis $p_1, \dots, p_n \in (\fot^{n})^*_{\ZZ}$ and its dual basis $q_1, \dots, q_n \in \fot^n_\ZZ$ as in Subsections \ref{symp-view} and \ref{comp-view}. 
We fix $\tilde{q}_i \in \fot^d_\ZZ$ such that $\pi(\tilde{q}_i)=q_i$ for $i=1, \dots, n$. 
Let $(x,[\theta])$ be the symplectic coordinate on $\mathring{X}$ and let $(w_1, \dots, w_n)$ be the complex coordinate on $\mathring{W}$ induced by $p_1, \dots, p_n \in (\fot^{n})^*_\ZZ$ respectively. 
If we write $p= \sum_{i=1}^n x_i p_i$, then, by (\ref{symp-coord}) and (\ref{comp-coord}) we have 
\begin{equation}
w_i (\chi_{g_{0}}(x, [\theta]))
= \prod_{j=1}^d (\sqrt{2l_j(p)} e^{ \sqrt{-1} \sum_{l=1}^n \langle u_j, \tilde{q}_l \rangle \theta_l})^{\langle p_i, r_j \rangle}
= e^{ (\frac{\partial g_{can}}{\partial x_i} + \sqrt{-1} \theta_i)}.
\end{equation}

In \cite{Ab1, Ab2}, Abreu demonstrated that all toric complex structures can be characterized by a symplectic potential that satisfies the conditions outlined in the subsequent definition.
\begin{definition}\label{def-po}
{\rm A function $g \in C^0 (P)$ is a {\em symplectic potential} if and only if it satisfies the following condition:
\newline
$(1)$ $g - g_{0} \in C^\infty(P)$,
\newline
$(2)$ The Hessian $\mathrm{Hess}_p g$ of $g$ at $p$ is positive definite for any $p \in \mathring{P}$,
\newline
$(3)$ there exists a strictly positive function $\beta \in C^\infty (P)$ such that 
$$
\det(\mathrm{Hess}_p g)=[\beta(p) \prod_{j=1}^d l_j(p)]^{-1} ~~, \forall p \in \mathring{P}.$$
The set of symplectic potentials is denoted by $\shS(P)$.}
\end{definition}
The results presented below originate from \cite{ Ab1,Ab2, Gui1,Gui2}, with additional contributions from \cite{BFMN} and a review provided in \cite{HK}.
\begin{theorem}\cite[Theorem 5.3]{HK}\label{com-sym}
Let $P \subset (\fot^{n})^* $ be a Delzant polytope.
Let $(X,\omega_{P})$  be a symplectic toric manifold and $(W_{P}, J)$ a complex toric variety constructed from $P$. 
Let $(L_{symp},h, \nabla)$ be a prequantum line bundle on $X$ and $(L_{comp},\bar{\partial}_{J})$ a holomorphic line bundle on $W_{P}$ constructed from $P$.
Fix a $\ZZ$-basis $p_1, \dots, p_n \in (\fot^{n})^*_\ZZ$. 
Let $(x,[\theta])$ be the symplectic coordinate on $\mathring{X}$ and $w=(w_1, \dots, w_n)$ the complex coordinate on $\mathring{W}$ induced by $p_1, \dots, p_n \in (\fot^{n})^*_{\ZZ}$ respectively. 
Then each $g \in \shS(P)$ defines a $T^n$-equivariant diffeomorphism $\chi_g : X \rightarrow W_{P}$ and a $T^n$-equivariant bundle isomorphism $\tilde{\chi}_g:  L_{symp} \rightarrow L_{comp}$ such that the following holds:  
\begin{enumerate}
\item The following diagram commutes: 
$$
\begin{tikzcd}
(L_{symp},h, \nabla) \arrow{rr}{\tilde{\chi}_g}\arrow{d}&& (L_{comp},\bar{\partial})\arrow{d} \\ 
(X,\omega_{P}) \arrow{rr}{\chi_g}  && (W_{P},J)
\end{tikzcd} 
$$
\item $(X,\omega_{P}, \chi_g^*J)$ is a K\"ahler manifold.

\item $\nabla$ is the Chern connection of the Hermitian holomorphic line bundle $(L_{symp}, h, \tilde{\chi}_g^* \bar{\partial_{J}})$.
\item $\chi_g|_{\mathring{X}}:  \mathring{X} \rightarrow \mathring{W}$ is a diffeomorphism given by 
\begin{equation}
w_i(\chi_g(x,[\theta]))=e^{ (\frac{\partial g}{\partial x_i}+ \sqrt{-1} \theta_i)}~~\text{$i=1, \dots, n$}.
\end{equation}
The map $\chi_g$ is independent of the choice of the basis $p_1, \dots, p_n \in (\fot^{n})^*_\ZZ$.
Moreover, if we write $w_i=e^{ (y_i+ \sqrt{-1}\theta_i)}$ for $i=1, \dots, n$, then the inverse mapping $(\chi_g|_{\mathring{X}})^{-1}: \mathring{W} \rightarrow \mathring{X}$ is given by 
\begin{equation}
x_i(\chi_g^{-1}(w))=\frac{\partial h}{\partial y_i}, ~~\theta_i((\chi_g^{-1})(w))=\theta_i ~~\text{for $i=1, \dots, n$,} 
\end{equation} 
where $h(y)=-g(x(y)) + \sum_{i=1}^n x_i(y)y_i$ is given by Legendre transformation.
\item
$\displaystyle \tilde{\chi}_g^* \mathring{s}_{\lambda} = e^{ (g-\sum_{i=1}^n x_i \frac{\partial g}{\partial x_i} )}\mathring{\mathbbm{1}}$ on $\mathring{X}$.
\end{enumerate}
 On the other hand, if $\chi:  X \rightarrow W_{P}$ is a $T^n$-equivariant diffeomorphism such that $(X,\omega_{P}, \chi^*J)$ is a K\"ahler manifold and that $\chi$ is homotopic to $\chi_{g_{0}}$, then there exists $g \in \shS(P)$ such that $\chi=\chi_g$. 
\end{theorem}
Finally, we review the work (\cite[Lemma 3.7]{BFMN}) of Baier, Florentino, Mour\~{a}o and Nunes, which will be used in the proof of convergence of quantum spaces in Theorem \ref{thm5-1}. 

\begin{lemma}\cite[Lemma 3.7]{BFMN} \label{conv}
For any $\psi$ strictly convex function in a neighbourhood of the moment polytope $P$ and any $m \in P_{\ZZ}$, the function: 
\begin{align*}
f_{m}: P & \rightarrow \RR\\
        x    &  \mapsto ^{t}(x-m)\frac{\partial \psi}{\partial x} - \psi(x)
\end{align*}
has a unique minimum at $x=m$ and 
$$
\lim_{t\rightarrow \infty} \frac{e^{-tf_{m}}}{\|e^{-tf_{m}}\|_{1}} \rightarrow \delta(x-m),$$
in the sense of distributions.
\end{lemma}

\section{Main results}

 \subsection{Quantum space $\shH_{k}$ associated to $\shP_{k}$}

$(*):$ 
Let $(X, \omega, J)$ be a 2n-dimensional toric variety determined by a Delzant polytope $P$ with moment map $\mu_{P}: X \twoheadrightarrow P \subset (\fot^{n})^{*}$.
 We consider the Hamiltonian $k$-dimensional subtorus action $\rho_{k}: T^{k} \rightarrow \Diff(X, \omega, J)$  with moment map $\mu_{k}: X \twoheadrightarrow \Delta_{k} \subset (\fot^{k})^{*}$. 
 Under assumption $(*)$, there exists a mixed polarization $\shP_{k}$ associated with $J$ and a Hamiltonian $k$-dimensional subtorus action by \cite[Theorem 3.8]{LW1}. From a physical perspective, the quantum space should be independent of the polarization choice. Specifically, the dimension of the quantum spaces should remain invariant. A result, often attributed to Danilov and Atiyah \cite{Da, GGK}, asserts that the number of integer points in $P$, which correspond to the images under $\mu_{P}$ of the Bohr-Sommerfeld fibers of the real polarization $\shP_{\RR}=\mathrm{Ker} d\mu$, equals the dimensionality of $H^{0}(X,L)$. Therefore, the dimension of the quantum space $\shH_{k}$ associated with $\shP_{k}$ is expected to equal the number of integer points in $P$, which matches the dimensions of the quantum spaces associated with both K\"ahler and real polarizations.

In this section, we will demonstrate that the dimension of $\shH_{k}$ equals the number of integer points in $P$ by showing that there is a natural basis for the quantum space $\shH_{k}$ indexed by the integer points in $P$ (see Theorem \ref{thm5-0}). 
Let $(L, h, \nabla)$ be the prequantum line bundle on $(X, \omega, J)$ determined by the moment polytope $P \in (\fot^{n})^*$ with curvature $F_{\nabla} = -i \omega$.  To address the polarization $\shP_{k}$ as defined previously, we extend the operator of covariant differentiation from smooth to distributional sections as in \cite{BFMN}. Consider the injection of smooth into distributional sections of $L$, determined by the Liouville measure, on any open set $U \subset X$,
 \begin{align*}
 i: \Gamma(U, L) & \rightarrow \Gamma_{c}(U,L^{-1})'\\
                         s & \mapsto i(s)(\phi)= \int_{U} \langle s, \phi \rangle e^{\omega} 
\end{align*} 
for any $\phi \in \Gamma_{c}(U,L^{-1})'$.
We extend the operator $\nabla_{\xi}$ on smooth sections to an operator on distributional sections, denoted by the same symbol $\nabla_{\xi}$, ensuring the following diagram remains commutative:
$$
\begin{tikzcd}
\Gamma(U, L) \arrow{r} {i} \arrow{d}{\nabla_{\xi}} &  (\Gamma_{c}(U, L^{-1}))' \arrow{d}{\nabla_{\xi}} \\
\Gamma(U, L)\arrow{r}{i} &  (\Gamma_{c}^{\infty} (U,L^{-1}))'.
\end{tikzcd}
$$
To determine $\nabla _{\xi} \delta$ for a general distributional section $\delta \in \Gamma_{c}(X, L^{-1})'$ not of the form $i(s)$, we define its transpose by integrating the operator $\nabla_{\xi}$ by parts. This yields, for any smooth section $s \in \Gamma(U, L)$ and smooth test section $\phi \in \Gamma_{c}(U, L^{-1})$,
\begin{equation}
 (\nabla_{\xi} i(s))(\phi) = \int_{U}\langle (\nabla_{\xi}s), \phi \rangle e^{\omega} =\int_{U} \langle s, -(div\xi \phi + \nabla^{-1} \phi)\rangle  e^{\omega}. 
 \end{equation}
Therefore, $\nabla^{}_{\xi}\delta$ can be characterized by its transpose:
\begin{equation}\label{eqcon}
(\nabla_{\xi} \delta)(\phi) = \delta(^{t}\nabla_{\xi} \phi),
\end{equation}
 for any $\phi \in \Gamma_{c}(U, L^{-1})$. Here, the transpose operator $^{t}\nabla_{\xi}$ is defined by
\begin{equation}^{t}\nabla_{\xi} \phi = - (div \xi \phi + \nabla_{\xi}^{-1} \phi).
\end{equation}
 Quantum spaces $\shH_{k}$ associated to polarizations $\shP_{k}$ are defined as
 \begin{equation}
 \shH_{k} = \{ \delta \in \Gamma_{c}(X, L^{-1})' \mid \nabla_{\xi} \delta=0, \forall~ \xi \in \Gamma(X, \shP_{k})\}.
 \end{equation}

For any $ q \in \Delta_{k}$, the level set $\mu_{k}^{-1}(q)$ is denoted by $X_{k}^{q}$ and forms a $T^{k}$-invariant coisotropic submanifold of $M$. We define the set of regular values of $\mu_{k}$ as $(\fot^{k})^{*}_{\mathrm{reg}}$. If $q \in (\fot^{k})^{*}_{\mathrm{reg}}$ and $T^{k}$ acts freely on $X^{q}_{k}$. Then the projection mapping
$$\pi: X^{q}_{k} \rightarrow X_{q,k}$$ is a principal $T^{k}$-fibration. There exists a unique symplectic form $\omega_{q,k}$ on $X_{q,k}$, such that $\pi^{*}\omega_{q,k} = \omega|_{X_{k}^{q}}$. Following \cite{LW2,LW3}, consider a connection $\alpha \in \Omega^{1}(X_{k}^{q}, \fot^{k})$ on $X^{q}_{k}$. The form  
\begin{equation}\pi^{*}(\frac{1}{(n-k)!}\omega_{q,k}^{n-k}) \wedge \alpha^{k}
\end{equation} is a volume form on $X_{k}^{q}$, 
denoted by $\vol_{k}^{q}$, where $\alpha^{k}$ is a $k$-form on $X^{q}_{k}$ defined by $\frac{\alpha \wedge \cdots \wedge \alpha}{v}$, with $v$ being a $T^{k}$-invariant top form on $\fot^{k}$. For general $q \in \Delta_{k}$, according to \cite[Theorem 5.9]{SL}, there exists an open dense subset $\check{X}_{k}^{q}$ of ${X}_{k}^{q}$, which also form a principal bundle over its quotient space. Notably, 
 $\check{X}_{k}^{q}=X_{k}^{q}$ when $q$ is a regular ralue. The volume form on $\check{X}_{k}^{q}$ is also denoted by $\vol_{k}^{q}$ by convention.

\begin{definition} \label{def3-1}For any $m\in P_{\ZZ}$, we define {\em the distributional section} $\delta_{k}^{m} \in \Gamma_{c}(X,L^{-1})'$ associated to $\sigma_{k,0}^{m}$ and $\mu_{k}$ by: for any test section $\phi \in \Gamma_{c}(X,L^{-1})$ $$\delta_{k}^{m} (\phi) =\frac{1}{ c^{m}_{k}} \int_{X^{q}_{k}}  \langle \sigma_{k,0}^{m}, \phi \rangle |_{X^{q}_{k}} \vol_{k}^{q},$$
where $c_{k}^{m}$ is a constant defined by: $c_{k}^{m} = \int_{X^{q}_{k}} \|\sigma_{k,0}^{m}|_{X^{q}_{k}}\| \vol_{k}^{q}.$
\end{definition}

\begin{theorem}\label{thm5-0}
Under assumption $(*)$, $\{\delta_{k}^{m}\}_{m\in P_{\ZZ}}$ forms a basis of quantum space $\shH_{k}$. In particular, $\mathrm{dim}\shH_{k}=\mathrm{dim}H^{0}(X,L)$.
\end{theorem} 

\begin{proof}

We first focus on demonstrating that $\delta^{m}_{k} \in \shH_{k}$. Let $q=i_{k}^{*}(m) \in \Im \mu_{k}$ and let $X_{k}^{q}=\mu_{k}^{-1}(q)$ be the level set with respect to $\mu_{k}$ and $q$. According to Sjamaar-Lerman's result in \cite{SL}, there is a stratification for $X^{q}_{k}$. We denote its maximal strata by $\check{X}_{k}^{q}$, which is an open dense subset of $X^{q}_{k}$ and is acted upon freely by $T^{k}$. 
Considering the definition of $\delta_{k}^{m}$, it is evident that $\delta_{k}^{m} \in \Gamma_{c}(X,L^{-1})'$ and, for any test section $\phi \in \Gamma(X,L^{-1})$,
\begin{equation}\label{eq3-0-15}
\delta_{k}^{m} (\phi) =\frac{1}{ c^{m}_{k}} \int_{X^{q}_{k}}  \langle \sigma_{k,0}^{m}, \phi \rangle |_{X^{q}_{k}} \vol_{k}^{q}=\delta_{k}^{m} (\phi) =\frac{1}{ c^{m}_{k}} \int_{\check X^{q}_{k}}  \langle \sigma_{k,0}^{m}, \phi \rangle |_{\check X^{q}_{k}} \vol_{k}^{q}.
\end{equation}
To establish that $\delta^{m}_{k} \in \shH_{k}$, it suffices to prove that $\nabla_{\xi} \delta_{k}^{m} = 0$ for any $\xi \in \Gamma(X, \shP_{k})$.

Take any test section $\phi \in \Gamma_{c}(X,L^{-1})$, according to equations (\ref{eqcon} and \ref{eq3-0-15}) and \cite[Proposition 3.7]{LW2}, we have:
\begin{equation}
(\nabla_{\xi} \delta_{k}^{m})(\phi) = \delta_{k}^{m}\left(^{t}{\nabla}_{\xi} \phi\right)=\int_{\check X_{k}^{q}}  \langle \nabla_{\xi}\sigma_{k,0}^{m}, \phi \rangle |_{\check X_{k}^{q}} \vol_{k}^{q}.
\end{equation}
Since $\sigma_{k,0}^{m}$ is $T^{k}$-invariant section, we have 
\begin{equation}
\nabla_{\xi}\sigma_{k,0}^{m}=0, \forall \xi \in \Gamma(X, \shI^{k}_{ \CC})
\end{equation} On the other hand $\nabla_{\xi} \sigma_{k,0}^{m}=0, \forall \xi \in \Gamma(X, \shD_{\CC}^{k} \cap \shP_{J})$, since $\sigma_{k,0}^{m}$ is holomorphic section with respect to $J$.  This implies 
\begin{equation}\nabla_{\xi}\sigma_{k,0}^{m} =0, \forall \xi \in \Gamma(X, \shP_{k}).
\end{equation}
We are then able to conclude that, for all $\phi \in \Gamma_{c}(X,L^{-1})$ and $\xi \in \Gamma(X, \shP_{k})$,
\begin{equation}
(\nabla_{\xi} \delta^{s})(\phi)=\int_{\check X_{k}^{q}}  \langle \nabla_{\xi}\sigma_{k,0}^{m}, \phi \rangle |_{\check X_{k}^{q}} \vol_{k}^{q}=0.
\end{equation} 
Therefore we have: $\delta^{m}_{k} \in \shH_{k}.$ It remains to show there are no more solution. By \cite[Theorem 3.2]{LW2}, we have, for any $\delta \in \shH_{k}$, $\mathrm{supp} \delta \subset \cup\mu_{k}^{-1}(q)$. Without loss of the generality, we assume $\mathrm{supp} \delta \subset \mu_{k}^{-1}(0)$. By \cite[Corollary 3.3]{LW2} , we have $\delta(\phi|_{X^{0}_{k}})=0$ for any $\phi \in \Gamma(X,L^{-1})$ satisfying $\phi|_{X^{0}_{k}}=0$. This implies the existence of $\underline\delta \in \Gamma(\check X^{0}_{k},(L^{0})^{-1})$ such that: 
\begin{equation}
\delta=\iota(\overline{\delta}),
\end{equation}
where $\iota: \Gamma(X^{0}_{k}, (L^{0})^{-1})' \rightarrow \Gamma(X,L^{-1})'$. According to the above discussion and by \cite[Theorem 3.6]{LW2}, we have:
\begin{equation}\nabla_{\xi} \underline{\delta}=0, \forall \xi \in \Gamma(\check X^{0}_{k}, \shP_{k}).\end{equation}
This implies $\underline{\delta}=i(s)$ for some smooth section on $\check{X}^{q}_{k}$ which satisfies $\nabla_{\xi}s=0$, for any $\xi \in \Gamma(X^{0}_{k},\shP_{k})$. Here $\iota(i(s)) \in \shH_{k}$ implies that $s$ is the restriction of the linear combination of holomorphic section of $L$.
It follows that for any $\delta$ is spanned by $\{\delta_{k}^{m}\}_{m\in P_{\ZZ}}$.
\end{proof}

Let $\shH_{k,\lambda}$ be weight $\lambda$ subspace of $\shH_{k}$. By \cite[Theorem 3.2]{LW2}, we have:

\begin{corollary}Under the assumption $(*)$, for any $\lambda \in \Im \mu_{k} \cap (\fot^{k})^{*}$, we have
\begin{equation}
\shH_{k,\lambda} = \{\delta_{k}^{m}\}_{i_{k}^{*}(m)= \lambda}.
\end{equation}
\end{corollary}

\begin{remark}
When $k=n$, the polarization $\shP_{n}$ is a real polarization and on an open dense subset of $M$, it coincides with $\shP_{\RR}= \ker d\mu$ studied in \cite{BFMN}. Additionally, $\shH_{n} \cong \shH_{\RR}$.
\end{remark}

\subsection{Degenerating K\"ahler polarizations to mixed polarizations $\shP_{k}$}\label{family}
From the perspective of physics, the quantum space should be independent of the choice of polarizations. 
In the previous section, under the assumption $(*)$, there exists singular mixed polarizations $\shP_{k}$ on $X$ as described in \cite{LW1}. Baier, Florentino, Mour\~ao, and Nunes \cite{BFMN} discovered a one-parameter family of K\"ahler polarizations that degenerates to the singular real polarization $\ker d\mu_P$, achieved by identifying a one-parameter family of symplectic potentials.  Motivated by \cite{BFMN}, to further explore the relationship between the K\"ahler polarization $\mathcal{P}_J$ and the mixed polarization $\mathcal{P}_k$, we seek to determine if there exists a one-parameter family of symplectic potentials such that the corresponding K\"ahler polarizations $\mathcal{P}_{k,t}$ degenerate to the singular mixed polarization $\mathcal{P}_k$ for $1 \leq k < n$. In this section, we construct a one-parameter family of K\"ahler polarizations $\mathcal{P}_{k,t}$, using a family of symplectic potentials $g_{k,t}$ (see equation \ref{eq-po-1}), inspired by constructions from \cite{BFMN} and \cite{HK}. We then demonstrate (see Theorem \ref{thm4-3}) that these K\"ahler polarizations $\mathcal{P}_{k,t}$ indeed degenerate to the mixed polarization $\mathcal{P}_k$.

According to the discussion in the last section, for any polarization $\shP_{k}$, there is a one-parameter family of K\"ahler polarizations, such that $$\lim_{t\rightarrow \infty}\shP_{k,t}=\shP_{k}.$$

Let $(X, \omega, J)$ be the toric with moment map $ \mu_{P}: X \twoheadrightarrow P$, where $P$ is Delzant polytope defined by $l_{j} \ge 0$ as (2.1), $j=1,\cdots,d$. Recall that the canonical complex structure $J$ can be written by a symplectic potential $g_{0}: P \rightarrow \RR$ as described in \cite{Ab1}\cite{Ab2}, and on $\mathring{X}$ under the symplectic coordinates $(x,\theta) \in \mathring{P} \times \RR^{n}/\ZZ^{n}$, the canonical complex structure and K\"ahler metric $\gamma= \omega(\cdot, J\cdot)$ are given by:

$$J = 
\begin{pmatrix}
0 & -G_{0}^{-1} \\
G_{0} & 0
\end{pmatrix}
; \gamma=\begin{pmatrix}
G_{0}& 0 \\
 0& G^{-1}_{0}
\end{pmatrix}
$$ 
where $G_{0} = \mathrm{Hess}g_{0}$ is the Hessian of $g_{0}$ such that $G_{0}(x)$ is positive definite on $\mathring{P}$ and satisfies the regularity conditions:
$$\det G_{0}(x)=[ \beta(x) \prod_{j=1}^{d}l_{j}(x)]^{-1},$$
for $\beta$ smooth and strictly positive on $P$.

Taking any $k$-dimensional subtorus $T^{k}$ of $T^{n}$, let $i_k^{*}: (\fot^{n})^{*} \rightarrow (\fot^{k})^{*}$ be the dual map of the inclusion of the Lie algebra $i_{k} : \fot^{k} \rightarrow \fot^{n}$. 
Under the assumption $(*)$, it can be seen  that:
 $\mu_{k} = i_{k}^{*} \circ \mu_{P}.$ Let $\varphi_{k}: \Delta_{k} \rightarrow \RR$ be a strictly convex function. We consider a one-parameter family of symplectic potentials 
 \begin{equation}\label{eq-po-1} g_{k,t} = g_{0}+ t(\varphi_{k}\circ i_{k}^{*}): P \rightarrow \RR. \end{equation}
Then the family of K\"ahler potentials $g_{k,t}$ determines a one-parameter family of toric complex structures denoted by $J_{k,t}$. 
Let $\shP_{k,t}$ be the K\"ahler polarizations associated to the complex structures $J_{k,t}$. We show that the K\"ahler polarizations $\shP_{k,t}$ converge to mixed polarization $\shP_{k}$ in the following theorem.
 
\begin{theorem}\label{thm4-3}
Under the assumption $(*)$, 
we have $$\lim_{t \rightarrow \infty} \shP_{k,t} =  \shP_{k},$$
 where the limit is taken in the positive Lagrangian Grassmannian of the complexified tangent space at each point in $X$.  
\end{theorem}

\begin{proof} 
 Let $i_{k}^{*}: (\fot^{n})^{*} \rightarrow (\fot^{k})^{*}$ be the dual map of the inclusion $i_{k}: \fot^{k} \rightarrow \fot^{n}$. Consider a basis of $(t^{n})^{*}_{\ZZ}$ denoted by $\tilde{p}_{1}, \cdots, \tilde{p}_{n}$, where $i^{*} (\tilde{p}_{1}), \cdots, i^{*}(\tilde{p}_{k})$ form a $\ZZ$ basis of $(t^{k})^{*}_{\ZZ}$ and $ i^{*}(\tilde{p}_{j}) = 0$ for $j=k+1, \cdots,n$. The basis $\tilde{p}_{1}, \cdots, \tilde{p}_{n}$ induces the complex coordinate $\tilde{w} =(\tilde{w}_{1},\cdots,\tilde{w}_{n})$ and symplectic coordinate $(\tilde{x}, \tilde{\theta})$ on the open dense subset $\mathring{X}$ of $X$. Moreover, by Theorem \ref{com-sym} (or \cite[Theorem 5.3]{HK}), the complex coordinate and symplectic coordinate are related by:
 \begin{equation} \label{eq-re} \tilde{w}_{k,t}^{i} = e^{2\pi (\frac{\partial g_{k,t}}{\partial{\tilde{x}_{i}} } + \sqrt{-1} \tilde{\theta}_{i})} ,   i=1, \cdots, n,\end{equation}
where $g_{k, t}$ is given by $g_{k,t} = g_{0}+t(   \varphi_{k} \circ i_{k}^{*} )$. This implies 

\begin{equation}
\frac{\partial g_{k,t}}{\partial \tilde{x}_{j}} =\frac{\partial g_{0}}{\partial \tilde{x}_{j}}, ~\text{and}~ \tilde{w}_{k,t}^{j}=\tilde{w}_{j},  \forall j= k+1,\cdots, n.  
\end{equation}

We therefore obtain that:
\begin{align*}
\shP_{k,t} &= \mathrm{span}_{\CC} \{\frac{\partial }{\partial \tilde{w}^{i}_{k,t}} , i=1,\cdots, n\} \\
        &= \mathrm{span}_{\CC}\{\frac{\partial }{\partial \tilde{w}^{i}_{k,t}}, i=1,\cdots, k\} \oplus \mathrm{span}_{\CC}\{  \frac{\partial }{\partial \tilde{w}^{j}_{k,t}}, j= k+1,\cdots,n\}\\
        &= \mathrm{span}_{\CC}\{\frac{\partial }{\partial \tilde{w}^{j}_{k,t}}, i=1,\cdots, k\} \oplus \mathrm{span}_{\CC}\{  \frac{\partial }{\partial \tilde{w}^{j}_{k,0}}, j= k+1,\cdots,n\}\\
         &= \mathrm{span}_{\CC}\{\frac{\partial }{\partial \tilde{w}^{j}_{k,t}}, i=1,\cdots, k\} \oplus \mathrm{span}_{\CC}\{  \frac{\partial }{\partial \tilde{w}_{j}}, j= k+1,\cdots,n\}.
\end{align*}

Recall that $\shD_{\CC}^{k} = \ker d\mu_{k} \otimes \CC$ is the complexification of $\ker d\mu_{k} \subset TX$, where $\mu_{k}: X \twoheadrightarrow \Delta_{k}$ is the moment map associated to the Hamiltonian $T^{k}$-action on $X$. Then, by direct computation, the distributions $\shD_{\CC}^{k}|_{\mathring{X}}$ are given by:
\begin{equation}\label{eq4-1}
\shD^{k}_{\CC} |_{\mathring{X}}= \mathrm{span}_{\CC} \{ \frac{\partial }{\partial \tilde{\theta}_{i}}, i=1,\cdots,n \} \oplus \mathrm{span}_{\CC} \{ \frac{\partial}{ \partial \tilde{x}_{j}}, j=k+1, \cdots,n\}.
\end{equation}\label{eq4-2}
Hence, on the open dense subset $\mathring{X}$, we have:
\begin{equation}\shD^{k}_{\CC} \cap \shP_{k,t} = \mathrm{span}_{\CC}\{  \frac{\partial }{\partial \tilde{w}^{j}_{k,0}}, j= k+1,\cdots,n\}= \shD^{k}_{\CC} \cap \shP_{J}.\end{equation}
 It is easy to see that $\shI^{k}_{\CC}|_{\mathring{X}} = \mathrm{span}_{\CC} \{ \frac{\partial }{\partial \tilde{\theta}_{i}}, i=1,\cdots,n \}$.
In combination with equations (\ref{eq4-1}) and (\ref{eq4-2}), it follows that
\begin{equation}
\shP_{k}|_{\mathring{X}}=(\shD^{k}_{\CC} \cap \shP_{J})\oplus \shI_{\CC}^{k}=\mathrm{span}_{\CC}\{  \frac{\partial }{\partial \tilde{w}_{j}}, j= k+1,\cdots,n\} \oplus \mathrm{span}_{\CC}\{\frac{\partial }{\partial \tilde{\theta}^{j}}, j=1,\cdots, k\}.  \end{equation}
 On open dense subset $\mathring{X}$, according to equation (\ref{eq-re}), it can be seen that
 \begin{equation}
 \frac{\partial }{\partial \tilde{w}^{j}_{k,t}}=\frac{\partial }{\partial {y}^{j}_{k,t}}-i\frac{\partial }{\partial \tilde{\theta}^{j}}= (G^{-1}_{k,t})_{jl}\frac{\partial}{\partial \tilde{x}_{l}} -i\frac{\partial }{\partial \tilde{\theta}^{j}},
 \end{equation}
 where  ${y}^{j}_{k,t}=\frac{\partial g_{k,t}}{\partial \tilde{x}_{j}}$ and $G_{k,t}^{-1}= (\mathrm{Hess}g_{k,t})^{-1}$. By the convexity of $\varphi_{k}$, we obtain: $(G^{-1}_{k,t})_{jl} \rightarrow 0$ as $t \rightarrow \infty$, for $j=1,\cdots, k$. 
This implies that 
 \begin{align*}
 \lim_{t\rightarrow \infty} \shP_{k,t}|_{\mathring{X}}   
       &=  \lim_{t\rightarrow \infty}\mathrm{span}_{\CC}\{\frac{\partial }{\partial \tilde{w}^{j}_{k,t}}, i=1,\cdots, k\} \oplus \mathrm{span}_{\CC}\{  \frac{\partial }{\partial \tilde{w}_{j}}, j= k+1,\cdots,n\}\\
                 &=  \lim_{t\rightarrow \infty}\mathrm{span}_{\CC}\{(G^{-1}_{k,t})_{jl}\frac{\partial}{\partial \tilde{x}_{l}} -i\frac{\partial }{\partial \tilde{\theta}^{j}}, j=1,\cdots, k\} \oplus \mathrm{span}_{\CC}\{  \frac{\partial }{\partial \tilde{w}_{j}}, j= k+1,\cdots,n\}\\
                 &=\mathrm{span}_{\CC}\{\frac{\partial }{\partial \tilde{\theta}^{j}}, j=1,\cdots, k\} \oplus \mathrm{span}_{\CC}\{  \frac{\partial }{\partial \tilde{w}_{j}}, j= k+1,\cdots,n\}=\shP_{k}|_{\mathring{X}}
 \end{align*}
Finally, according to \cite[Lemma 4.2]{BFMN}, we have $\shD^{k}_{\CC}  \cap \shP_{k,t}  = \shD^{k}_{\CC} \cap \shP_{J}$ and $ \lim_{t\rightarrow \infty} \shP_{k,t}=\shP_{k}$ on the point outside the open orbit $\mathring{X}$. To summarize the above discussion, we have:
 $$ \lim_{t\rightarrow \infty} \shP_{n,t}=\shP_{n}.$$
 \end{proof}
 
\begin{remark}When $k=n$, the family of K\"ahler polarizations $\shP_{n,t}$ coincide with the family of K\"ahler polarizations $\shP_{\CC}^{s}$ studied in \cite[Theorem 1.2]{BFMN}. But our result stated that not only $\Gamma(X,  \lim_{t\rightarrow \infty} \shP_{n,t})=\Gamma(X, \shP_{n})$ but as distribution $ \lim_{t\rightarrow \infty} \shP_{k,t}=\shP_{k}.$
\end{remark}\label{re4-3}
\begin{remark} The symplectic potential $g_{k,t}$ is defined by:
$$g_{k,t} = g_{0}+ t(\varphi_{k}\circ i_{k}^{*}): P \rightarrow \RR.$$
It follows $g_{k,0}=g_{0}$ for any $1\le k \le n$. Therefore the one-parameter family of K\"ahler polarizations $\shP_{k,t}$ relate K\"ahler polarization $\shP_{J}=\shP_{k,0}$ to the polarization $\shP_{k}$.  
\end{remark}

\subsection{Compatibility with imaginary-time flow}
 In the previous subsection, we identified a family of K\"{a}hler polarizations that connects 
K\"{a}hler polarization $\shP_{J}$ with mixed polarization $\shP_{k}$ through symplectic 
potentials. Particularly, according to Abreu's work in \cite{Ab2}, the complex structure of 
this family exhibits a $T^n$-invariant toric complex structure.

Using the imaginary-time flow method, as employed by Mour\~{a}o and Nunes in \cite{MN}, 
Leung and Wang construct a one-parameter family of K\"{a}hler polarizations in \cite{LW1}. 
These polarizations converge to a mixed polarization on a K\"{a}hler manifold with $T$-symmetry. 
Specifically, by Theorem 3.19 in \cite{LW1}, the imaginary time flow method identifies a 
family of complex structures on the toric variety $(X, \omega, J)$, such that the corresponding 
K\"{a}hler polarizations converge to the mixed polarization $\shP_{k}$. In this subsection, we will demonstrate the conditions under which the K\"{a}hler polarizations of the two families coincide (see Theorem \ref{thm3-0-4}).

Under assumption $(*)$, according to Theorem \ref{thm4-3}, we know that there is a one-parameter family of K\"ahler polarizations \{$\shP_{k,t}$\} determined by $g_{k,t} = g_0 + \varphi_{k}\circ i^{*}_{k}$ converging to $\shP_{k}$, where $i_{k}^{*}: (\fot^{n})^{*} \rightarrow (\fot^{k})^{*}$ is the dual map of the inclusion $i_{k}: \fot^{k} \rightarrow \fot^{n}$, and $\varphi_{k}:\Delta_{k} \rightarrow \RR$ is a strictly convex function.
Let $\hat{\varphi}_{k}: \Delta_{k} \rightarrow \RR$ be another strictly convex function. Let $X_{\hat{\varphi}_{k}}$ be the Hamiltonian vector field on $X$ associated to the function $\hat{\varphi}_{k}\circ i_{k}^{*} \circ \mu_{P}$. 
 According to \cite[Theorem 3.19, 3.20, and 3.21]{LW1}, we are able to construct a one-parameter family of K\"{a}hler polarizations $\shP_{k,t}'$ using imaginary-time flow $e^{-itX_{\hat{\varphi}_{k}}}$. 
 
We will demonstrate that when $\varphi_{k}=\hat{\varphi}_{k}$ and $\shP_{k,0}=\shP_{k,0}'$, these two one-parameter families of K\"{a}hler polarizations coincide. This equivalence is established by applying \cite[Theorem 3.2]{LW3} and Legendre tranformation described in \cite{Ab2} to our current setting.

\begin{theorem}\label{thm3-0-4}
Let $\{\mathcal{P}_{k,t}\}$ be a one-parameter family of K\"{a}hler polarizations given by the symplectic potentials $g_{k,t} = g_0 + \varphi_k \circ i_k^*$, and let $\{\mathcal{P}'_{k,t}\}$ be another one-parameter family of K\"{a}hler polarizations derived using the imaginary-time flow $e^{-itX_{\hat{\varphi}_k}}$. If $\shP_{k,0}=\shP_{k,0}'$ and $\varphi_{k}=\hat{\varphi}_{k}$, then
$$\mathcal{P}_{k,t} = \mathcal{P}'_{k,t}, ~\forall t\ge 0.$$
\end{theorem}

\begin{proof} 
Since $\varphi_{k}$ is a strictly convex function on $\Delta_{k}$, by \cite[Proposition 6.5]{HK}, there is a one-parameter family of toric complex structure \{$J_{k,t}$\} defined by $g_{k,t} = g_{k,0} + \varphi_k \circ i_k^*$.
In fact, \{$\shP_{k,t}$\} is the K\"ahler polarizations associated to $J_{k,t}$. Additionally, \cite[Theorem 3.19]{LW1} shows the existence of a one-parameter family of complex structure \{$J_{k,t}'$\} determined by the imaginary-time flow $e^{-itX_{\varphi_k}}$.
According to \cite[Theorem 3.20]{LW1}, $(M, \omega, J_{k,t})'$ is a K\"ahler manifold and $\{\mathcal{P}'_{k,t}\}$ is the one-parameter family of K\"ahler polarization associated to \{$J_{k,t}'$\}. Therefore, to prove $\mathcal{P}_{k,t} = \mathcal{P}'_{k,t},$ it is sufficient to demonstrate $J_{k,t}=J_{k,t}',  ~\forall t\ge 0.$ By assumption $J_{k,0}'=J_{k,0}$, it follows that $J_{k,0}'$ is a toric complex structure. It is easy to verify that $J_{k,t}'$ is also toric complex structure. Let $g_{k,t}'$ be the potential function associated to $J_{k,t}'$.  Based on Abreu's \cite[Theorem 2.8]{Ab2}, in order to prove $J_{k,t}=J_{k,t}' $, it's enough to demonstrate $g_{k,t}'=g_{k,t}, ~\forall t\ge 0$. 
Let $h_{k,t}'$ be the K\"ahler potential with respect to $J_{k,t}'$ on the open dense subset $\mathring{X}$. By \cite[Theorem 5.3]{HK} (or Legendre transformation in  \cite{Ab2})
), $h_{k,0}'$ is given by 
\begin{equation}\label{eq3-1-1}
h_{k,0}' = -g_{k,0}'(x)+\sum x \cdot y_{k,0}.
\end{equation}
Here $x$ is the symplectic coordinate and $y_{k,0}$ is the complex coordinate on $\mathring{X}$ described as before.
By \cite[Theorem 3.2 ]{LW3}, 
\begin{equation}\label{eq3-1-2}
2h_{k,t}'=2h_{k,0}' -2t\varphi_{k}+2t \beta(X_{\varphi_{k}}),
\end{equation}
 where $\beta$ is the real local potential for $\omega$ defined by $\beta=\mathrm{Re}(i\bar{\partial}_{k,0}h_{k,0})$ and $\bar{\partial}_{k,0}$ being the $\bar{\partial}$-operator with respect to the complex structure $J_{k,0}$.
By Burns-Guillemin's results in \cite{BG} (or see Theorem \ref{thm2-0}), we have:
\begin{equation}\label{eq3-1-3}
\beta=\mathrm{Re}(i\bar{\partial}_{k,0}h_{k,0})=\mathrm{Re}(i\bar{\partial}_{k,0}h_{k,0})=\sum_{j}^{n}x_{j}d\theta_{j},
\end{equation}
where $(x,\theta)$ is the symplectic coordinate on $\mathring{X}$. Under the inclusion $i_{k}: \fot_{k}\rightarrow \fot_{n}$, by direct computation, we have:
\begin{equation}\label{eq3-1-4}
X_{\varphi_{k}}=\sum_{j=1}^{k}\frac{\partial \varphi_{k}}{\partial x_{j}}\frac{\partial}{\partial \theta_{j}}.
\end{equation}
By combining equations (\ref{eq3-1-1}), (\ref{eq3-1-2}), (\ref{eq3-1-3}), and (\ref{eq3-1-4}), we obtain:
\begin{equation}\label{eq3-1-5}
2h_{k,t}'=2(-g_{k,0}'(x)+\sum x \cdot y_{k,0}) -2t\varphi_{k}+2t\sum_{j=1}^{k}\frac{\partial \varphi_{k}}{\partial x_{j}}x_{j}.
\end{equation}
Then, by Legendre transformation (see \cite{Ab2}), we have:
\begin{equation}
g_{k,t}'= g_{k,0}' + t \varphi_{k} \circ i_{k}^{*}.
\end{equation}
Based on the assumption $J_{k,0}=J_{k,0}'$, it follows that $g_{k,0}=g_{k,0}'$. Consequently, we can conclude that $g_{k,t}=g_{k,t}'$ for all $t \ge 0$.
\end{proof}

\begin{corollary}\label{com-family}
Let $\{J_{k,t}\}$ be the one-parameter family of toric complex structures given by the potential $g_{k,t} = g_0 + \varphi_k \circ i_k^*$, and let $\{J'_{k,t}\}$ be another one-parameter family of complex structures derived using imaginary-time flow $e^{-itX_{\hat{\varphi}_k}}$. If $\shP_{k,0}=\shP_{k,0}'$ and $\varphi_{k}=\hat{\varphi}_{k}$, then $$J_{k,t} = J'_{k,t}, ~\forall t\ge 0.$$
\end{corollary}

\begin{remark} This implies that $J'_{k,t}$ given by imaginary-time flow is also toric complex structures. Additionally,
combining Theorem \ref{thm3-0-4} and \cite[Theorem 3.21]{LW1} provides an alternative proof of Theorem \ref{thm4-3}.
Moreover, for any $t \ge 0$, $(X,\omega, J_{k,t})$ is a K\"ahler manifold. Moreover, the path of K\"ahler metrics $h_{k,t}=\omega(-,J_{k,t}-)$ is a complete geodesic ray in the space of K\"ahler metrics of $X$.
\end{remark}

\subsection{Large limit of quantum spaces}\label{quant}
$(*):$ 
Let $(X, \omega, J)$ be a 2n-dimensional toric variety determined by a Delzant polytope $P$ with moment map $\mu_{P}: X \twoheadrightarrow P \subset (\fot^{n})^{*}$.
 We consider the Hamiltonian $k$-dimensional subtorus action $\rho_{k}: T^{k} \rightarrow \Diff(X, \omega, J)$  with moment map $\mu_{k}: X \twoheadrightarrow \Delta_{k} \subset (\fot^{k})^{*}$. 
This section is dedicated to exploring the correlation between the quantum space $\shH_{k,0}=\shH_{J}$ and $\shH_{k}$, which are 
associated with the K\"ahler polarization $\shP_{J}$ and mixed polarization $\shP_{k}$, respectively. Consider a strictly convex function $\varphi_{k}: \Delta_{k} \rightarrow \RR$. As discussed in the previous section, we examine a one-parameter family of symplectic potentials $g_{k,t} = g_{0} + t(\varphi_{k} \circ i_{k}^{*} )$. This investigation leads us to Theorem \ref{thm4-3}, where we identify a one-parameter family of K\"ahler polarizations $\shP_{k,t}$, which converge to $\shP_{k}$. 
This convergence establishes a vital link between the initial K\"ahler polarization $\shP_{k,0}$ and the mixed polarization $\shP_{k}$.
Furthermore, for any $m \in P \cap (\fot^{n})^{*}_{\ZZ}$, a holomorphic section $\sigma^{m}$ of $L_{comp}$ is defined by equation (\ref{def-ho-m}). The set $\{\sigma^{m}\}_{m \in P \cap (\fot^{n})^{*}_{\ZZ}}$ forms a basis of $H^{0}(W_{P}, L_{comp})$. According to the equations \ref{eq-ac}, \ref{def-ho-m}, and Theorem \ref{com-sym}, the transformed sections
$\{\sigma^{m}_{k,t}:=(\tilde{\chi}_{g_{k,t}})^*\sigma_{m} \}_{m \in (\fot^{n})^{*}_{\ZZ}}$  form a basis of the space $\shH_{k,t}$ as described by:
\begin{equation}\label{eq-inv}\shH_{k,t}= \mathrm{span}\{\sigma^{m}_{k,t}:=(\tilde{\chi}_{g_{k,t}})^*\sigma_{m} \}_{m \in (\fot^{n})^{*}_{\ZZ}}.\end{equation}

Our goal is to elucidate the relationship between the quantum spaces $\shH_{k,0}$ and $\shH_{k}$ across the one-parameter family of K\"ahler polarizations $\{\shP_{k,t}\}_{t\ge0 }$. Specifically, we aim to  demonstrate that 
``$\lim_{t\rightarrow \infty}\shH_{k,t}=\shH_{k}$" 
in the following sense. 
\begin{theorem} \label{thm5-1} 
Under assumption $(*)$, for any strictly convex function $\varphi_{k}$ in a neighborhood of $\Delta_{k}$ and $m \in P \cap \fot^{*}_{\ZZ}$, consider the family of $L^{1}$-normalized $J_{k,t}$-holomorphic sections $\frac{\sigma_{k,t}^{m}}{\|\sigma_{k,t}^{m}\|_{L_{1}}}$, under the injection $i: \Gamma(X, L) \rightarrow \Gamma_{c}(X,L^{-1})'$.
We observe that:
\begin{equation}
i\left(\frac{\sigma_{k,t}^{m}}{\|\sigma_{t}^{m}\|_{L_{1}}}\right) \longrightarrow \delta_{k}^{m}, \text{as}~ t \rightarrow \infty.
\end{equation}
In the sense that for any $\phi \in \Gamma_{c}(X, L^{-1})'$, we have:
\begin{equation}
\lim_{t \rightarrow \infty} \int_{X} \langle \frac{\sigma_{k,t}^{m}}{\|\sigma_{k,t}^{m}\|_{L_{1}}}, \phi \rangle e^{\omega} = \frac{1}{c_{k}^{m}}\int_{X_{k}^{q}} \langle \sigma_{k,0}^{m}, \phi \rangle |_{\mu_{k}^{-1}(0)} \vol_{k}^{q},
\end{equation}
where $q=i_{k}^{*}(m)$ and
$c_{k}^{m} = \int_{X^{q}_{k}} \|\sigma_{k,0}^{m}|_{X^{q}_{k}}\| \vol_{k}^{q}.$
\end{theorem}

\begin{proof}

 Recall that $i_{k}^{*}: (\fot^{n})^{*} \rightarrow (\fot^{k})^{*}$ is the dual Lie algebra homomorphism of $i_{k}: \fot^{k} \rightarrow \fot^{n}$. We observe that $i_{k}^{*}(P) = \Delta_{k}$ due to $\mu_{k} = i_{k}^{*} \circ \mu_{p}$. Now we take a basis of $(\fot^{n})^{*}_{\ZZ}$ denoted by $\tilde{p}_{1}, \cdots, \tilde{p}_{n}$ such that $i_{k}^{*} (\tilde{p}_{1}), \cdots, i_{k}^{*}(\tilde{p}_{k})$ form a $\ZZ$ basis of $(t^{k})^{*}_{\ZZ}$ and $ i_{k}^{*}(\tilde{p}_{j}) = 0$ for $j=k+1, \cdots,n$. Then $\tilde{p}_{1},\cdots, \tilde{p}_{n} $ induces the coordinate $(\tilde{x}_{1},\cdots,\tilde{x}_{k})$ on $(\fot^{k})^{*}$. Without loss of generality, we assume $i_{k}^{*}(m)=0$, i.e. $q=0$. If we write $m =\sum_{i=1}^n \tilde{m}_i \tilde{p}_i \in (\fot^{n})^*_{\ZZ}$, according to (\ref{def-sigma-m}), Theorem \ref{com-sym}, and Theorem \ref{thm3-0-4}, we have:
\begin{equation}\label{eq4-0-3}
(\sigma_{k,t}^{m})=(\tilde{\chi}_{k,t}^* \sigma^m) = (\tilde{\chi}_{k,t}^* \{ (\prod_{i=1}^n (\tilde{w}_i)^{\tilde{m}_i}) s_{\lambda} \})=e^{-t\alpha_{m}}(\prod_{i=1}^n (\tilde{w}_i)^{\tilde{m}_i}) s_{\lambda},
\end{equation}
where $\alpha_m(\tilde{x}) =\sum_{i=1}^n (\tilde{x}_i -\tilde{m}_i)\frac{\partial (\varphi_{k} \circ i_{k}*)}{\partial \tilde{x}_i}(\tilde{x})-(\varphi_{k} \circ i_{k}^*)(\tilde{x}) .$

 for any test section $\phi \in \Gamma_{c}(X, L^{-1})$, we have:
\begin{align*}
 \int_{X} \langle \frac{\sigma_{k,t}^{m}}{\|\sigma_{k,t}^{m}\|_{L_{1}}}, \phi \rangle e^{\omega} &= \int_{X} \langle \frac{\sigma_{k,t}^{m}}{\|\sigma_{k,t}^{m}\|_{L_{1}}}, \phi \rangle e^{\omega} \\
 &=  \frac{1}{\|\sigma_{k,t}^{m}\|_{L_{1}}} \int_{X} \langle \sigma_{k,t}^{m}, \phi \rangle e^{\omega} \\
 &= \frac{\|e^{-t\alpha_{m}}\|_{L_{1}}}{\|\sigma_{k,t}^{m}\|_{L_{1}}}  \int_{X} \frac{e^{-t\alpha_{m}}}{\|e^{-t\alpha_{m}}\|_{L_{1}}}  \langle \sigma_{k,0}^{m}, \phi \rangle e^{\omega}.
\end{align*}


By equation (\ref{eq4-0-3}), we have:

\begin{equation}\label{eq4-0-4}
 \frac{\|\sigma_{k,t}^{m}\|_{L_{1}}}{\|e^{-t\alpha_{m}}\|_{L_{1}}} = \frac{1}{\|e^{-t\alpha_{m}}\|_{L_{1}}} \int_{X} |\sigma_{k,t}^{m}| e^{\omega} 
= \frac{1}{\|e^{-t\alpha_{m}}\|_{L_{1}}} \int_{X}e^{- t\alpha_m }|\sigma_{k,0}^{m}|e^{\omega}   
\end{equation}

Let $Q^{m}_{k}(\tilde{x}_{1}, \cdots, \tilde{x}_{k}) d\tilde{x}_{1} \wedge \cdots \wedge d\tilde{x}_{k}$ is the $k$-form determined by fiber integral the $n$-form $|\sigma_{k,0}^{m}|e^{\omega}$ along fiber $\mu^{-1}(\tilde{x}_{1}, \cdots, \tilde{x}_{k})$.
 By equation (\ref{eq4-0-4}) and Lemma \ref{conv} , we obtain:

\begin{align*}
\lim_{t \rightarrow \infty}  \frac{\|\sigma_{k,t}^{m}\|_{L_{1}}}{\|e^{-t\alpha_{m}}\|_{L_{1}}}&=\lim_{t \rightarrow \infty}  \int_{X}\frac{e^{-t\alpha_{m}}}{\|e^{-t\alpha_{m}}\|_{L_{1}}}|\sigma_{k,0}^{m}|e^{\omega}\\
&=  \lim_{t \rightarrow \infty} \int_{\Delta_{k}}\frac{e^{-t\hat{\alpha}_{m}}}{\|e^{-t\alpha_{m}}\|_{L_{1}}}  Q^{m}_{k}(\tilde{x}_{1}, \cdots, \tilde{x}_{k})d\tilde{x}_{1} \wedge \cdots \wedge d\tilde{x}_{k} = Q_{k}^{m}(0,\cdots,0)\\
&
= \int_{X_{k}^{0}} |\sigma_{k,0}^{m}| \vol_{k}^{0}.
\end{align*}
We denote 
\begin{equation}\label{eq4-0-5} \lim_{t \rightarrow \infty}  \frac{\|\sigma_{k,t}^{m}\|_{L_{1}}}{\|e^{-t\alpha_{m}}\|_{L_{1}}}=:c_{k}^{m} \in \RR\end{equation}
Similarly, we have:

\begin{equation}\label{eq4-0-6}
 \lim_{t\rightarrow \infty}  \int_{X} \frac{e^{-t\alpha_{m}}}{\|e^{-t\alpha_{m}}\|_{L_{1}}}  \langle \sigma_{k,0}^{m}, \phi \rangle e^{\omega}=\int_{X_{k}^{0}} \langle \sigma_{k,0}^{m}, \phi \rangle\vol_{k}^{0}.
\end{equation}

Combine equations (\ref{eq4-0-3}), (\ref{eq4-0-5})and (\ref{eq4-0-6}), we obtain:
\begin{align*}
\lim_{t \rightarrow \infty}  \int_{X} \langle \frac{\sigma_{k,t}^{m}}{\|\sigma_{k,t}^{m}\|_{L_{1}}}, \phi \rangle e^{\omega} &
 = \lim_{t \rightarrow \infty}  \frac{1}{\|\sigma_{k,t}^{m}\|_{L_{1}}} \int_{X} \langle \sigma_{k,t}^{m}, \phi \rangle e^{\omega} \\
 &= \lim_{t \rightarrow \infty}  \frac{\|e^{-t\alpha_{m}}\|_{L_{1}}}{\|\sigma_{k,t}^{m}\|_{L_{1}}}  \int_{X} \frac{e^{-t\alpha_{m}}}{\|e^{-t\alpha_{m}}\|_{L_{1}}}  \langle \sigma_{k,0}^{m}, \phi \rangle e^{\omega}\\
 &= \lim_{t \rightarrow \infty}   \frac{\|e^{-t\alpha_{m}}\|_{L_{1}}}{\|\sigma_{k,t}^{m}\|_{L_{1}}}  \lim_{t \rightarrow \infty} \int_{X} \frac{e^{-t\alpha_{m}}}{\|e^{-t\alpha_{m}}\|_{L_{1}}}  \langle \sigma_{k,0}^{m}, \phi \rangle e^{\omega}\\
 &= \frac{1}{c_{k}^{m}} \int_{X^{0}_{k}}\langle \sigma_{k,0}^{m}, \phi \rangle \vol_{k}^{0}=(\delta_{k}^{m}, \phi).
\end{align*}
\end{proof}
\begin{remark}
The authors in \cite{BFMN} studied the real polarization $\shP_{\RR}$, defined by $\ker d \mu$. When $k = n$, $\shP_{k}$ coincides with $\shP_{\RR}$ on the open dense subset $\mathring{X}$ of $X$. Furthermore, the above theorem aligns with the results stated in \cite[Theorem]{BFMN}.
\end{remark}

\section{Appendix}
\subsection{Polarizations on symplectic manifolds}\label{ap-pol}

A step in the process of geometric quantization is to choose a polarization. We first recall the definition of distribution and polarization.

\begin{definition} \label{def4-1}A {\em complex distribution} $\shP$ on a manifold $M$ is a sub-bundle of the complexified tangent bundle $TM\otimes \CC$, such that for every $x \in M$, the fiber $\shP_{x}$ is a subspace of $T_{x}M\otimes \CC $.
\end{definition}

We denote by $\Gamma(M,\shP)$ the space of all vector fields on $M$ that are tangent to $\shP$. Here a vector field $u$ of $M$ is tangent to $\shP$ if for every $x \in M$, $u_{x} \in \shP_{x}$. Then we can define complex polarizations on symplectic manifolds. 

\begin{definition}\label{def4-2} A {\em complex polarization} $\shP$ of a symplectic manifold $(M,\omega)$ is a complex distribution $\shP \subset TM\otimes \CC$ satisfying the following conditions:
\begin{enumerate} [label = (\alph*)]
\item $\shP$ is involutive, i.e. if $u,v \in \Gamma(M,\shP)$, then $ [u,v] \in \Gamma(M,\shP)$;
\item for every $x \in M$, $\shP_{x} \subseteq T_{x}M \otimes {\CC}$ is Lagrangian; and
\item $\rank(\shP \cap \overline{\shP} \cap TM)$ is constant.
\end{enumerate}
\end{definition}

 \begin{remark} Let $\shP$ be a complex polarization on a symplectic manifold $(M, \omega)$, then $\shP$ is called:
 \begin{enumerate}
 \item {\em real polarization}, if $\shP = \overline{\shP}$;
 \item {\em K\"ahler polarization}, if $\shP \cap \overline{\shP} = 0$;
  \item {\em mixed polarization}, if $0 < \rank(\shP \cap \overline{\shP} \cap TM) < n$.
 \end{enumerate}
 \end{remark}

 To study the singular polarization on toric manifolds, we recall the definitions of singular polarizations and smooth section of singular polarizations on symplectic manifolds $M$ as follows. 
\begin{definition}\label{def4-3}
 $\shP \subset  TM \otimes\CC$ is a {\em singular complex distribution} on $M$ if it satisfies:
$\shP_{p} $ is a vector subspace of $ T_{p}M \otimes \CC$, for all points $p \in M$.
\end{definition}

\begin{definition}\label{def4-4}
Let $\shP$ be a singular complex distribution of $TM\otimes \CC$. For any open subset $U$ of $M$, {\em the space of smooth sections of $\shP$ on $U$} is defined by the smooth section of $TM\otimes \CC$ with value in $\shP$, that is,
$$ \Gamma(U, \shP) = \{ v \in \Gamma(U, TM \otimes \CC) \mid  v_{p} \in \shP_{p}, \forall  p\in U \}.$$
In particular, 
$$\Gamma(M, \shP) = \{ v \in \Gamma(M, TM \otimes \CC) \mid  v_{p} \in \shP_{p},  \forall  p\in M \}.$$
\end{definition}

\begin{remark}
In this paper, we only consider such distributions with mild singularities in the sense that they are only singular outside an open dense subset $\check{M}\subset M$. Under our setting, we define the {\em involutive distribution}  as follows
\end{remark}

\begin{definition}\label{def4-5}
Let $\shP$ be a singular complex distribution on $M$. $\shP$ is {\em involutive} if it satisfies,
$$[u,v] \in \Gamma(M,\shP), \mathrm{~for~ any~} u,v \in \Gamma(M, \shP)$$
\end{definition}

\begin{definition}\label{def4-6}
Let $\shP \subset TM \otimes\CC$ be a singular complex distribution on $M$. $\shP$ is called {\em smooth on $\check{M}$} if 
 $\shP|_{\check{M}}$ is a smooth sub-bundle of the tangent bundle $T\check{M} \otimes \CC$.
\end{definition}

\begin{definition}\label{def4-7} Let $\shP$ be a singular complex distribution $\shP$ on $M$ and smooth on $\check{M}$. $\shP$ is called a {\em singular polarization on $M$ and smooth on $\check{M}$}, if it satisfies the following conditions:
\begin{enumerate} [label = (\alph*)]
\item $\shP$ is involutive, i.e. if $u,v \in \Gamma(M,\shP)$, then $ [u,v] \in \Gamma(M,\shP)$;
\item for every $p \in \check{M}$, $\shP_{p} \subseteq T_{p}M \otimes \CC$ is Lagrangian; and
\item  $\rank(\shP \cap \overline{\shP} \cap TM)|_{\check{M}}$ is a constant.
\end{enumerate}
\end{definition}

\begin{definition} \label{def4-8}
 If $\shP$ is a singular polarization on $M$ and smooth on $\check{M}$, and $\shP|_{\check{M}}$ is a smooth polarization on $\check{M}$, then $\shP$ is called 
 \begin{enumerate}
 \item a {\em real polarization}, if $\shP = \overline{\shP}$ on $\check{M}$;
 \item a {\em mixed polarization}, if $0 < \rank(\shP \cap \overline{\shP} \cap TM)|_{\check{M}} < n$.
 \end{enumerate}
\end{definition}

\newpage
\phantomsection
\vspace*{40pt}

\bibliographystyle{amsplain}

\end{document}